\documentclass[12pt]{amsart}
\usepackage[lite]{amsrefs}
\usepackage{amssymb, amsmath, enumitem}
\usepackage{comment}
\newcommand{\R}{\mathbb{R}}

\newcommand{\N}{\mathbb{N}}

\newcommand{\M}{\mathcal{M}}
\newcommand{\E}{\mathcal{E}}
\newcommand{\A}{\mathcal{A}}
\newcommand{\LL}{\mathcal{L}}

\newcommand{\I}{{\bf{I}}}
\newcommand{\g}{{\bf{G}}}
\numberwithin{equation}{section}
\usepackage[linktocpage]{hyperref}
\hypersetup{colorlinks=true, citecolor=red, linkcolor=blue, urlcolor=blue}
\theoremstyle{plain}
\newtheorem{Thm}{Theorem}[section]
\newtheorem{Cor}[Thm]{Corollary}
\newtheorem{Lem}[Thm]{Lemma}
\newtheorem{Prop}[Thm]{Proposition}

\theoremstyle{definition}
\newtheorem{Def}[Thm]{Definition}
\newtheorem{Rem}[Thm]{Remark}

\theoremstyle{remark}

\begin{document}
%%%%%%%%%%%%%%%%%%%%%%%%%%%%%%%%%%%% 
\title[Solutions to sublinear elliptic equations]{Solutions to sublinear elliptic equations
with finite generalized energy}

\author{Adisak Seesanea}
\address{Department of Mathematics, Hokkaido University, Sapporo, Hokkaido 060-0810, Japan}
\email{\href{mailto:seesanea@math.sci.hokudai.ac.jp}{seesanea@math.sci.hokudai.ac.jp}}

\author{Igor E. Verbitsky}
\address{Department of Mathematics, University of Missouri, Columbia, MO 65211, USA}
\email{\href{mailto:verbitskyi@missouri.edu}{verbitskyi@missouri.edu}}
%%%%%%%%%%%%%%%%%%%%%%%%%%%%%%%%%%%%
\begin{abstract}
We give necessary and sufficient conditions for the existence of a positive solution
with zero boundary values to the elliptic equation
\[
\LL u = \sigma u^{q} + \mu \quad \text{in} \;\; \Omega,
\]
in the sublinear case $0<q<1$, with finite generalized energy: 
$\mathbb{E}_{\gamma}[u]:=\int_{\Omega} |\nabla u|^{2} u^{\gamma-1}dx<\infty$,
for $\gamma >0$.  In this case $u \in L^{\gamma+q}(\Omega, \sigma)\cap L^{\gamma}(\Omega, \mu)$, 
where $\gamma=1$ corresponds to finite energy solutions.

Here $\LL u:= -\,\text{div}(\A\nabla u)$ is a linear uniformly elliptic operator with 
bounded measurable coefficients, and $\sigma$, $\mu$ are nonnegative 
 functions (or Radon measures),  
on an arbitrary domain $\Omega\subseteq \R^n$ which possesses a positive 
Green function associated with $\LL$. 

When $0<\gamma\le 1$, this result yields sufficient conditions
for the existence of a positive solution to the above problem which belongs to
the Dirichlet space $\dot{W}_{0}^{1,p}(\Omega)$ for $1<p\le 2$. 
\end{abstract}
%%%%%%%%%%%%%%%%%%%%%%%%%%%%%%%%%%%%
\thanks{A. S. is partially supported by Development and Promotion of Science and Technology Talents Project, Thailand (DPST)}
\maketitle
%%%%%%%%%%%%%%%%%%%%%%%%%%%%%%%%%%%%
%%%%%%%%%%%%%%%%%%%%%%%%%%%%%%%%%%%%
\section{Introduction}\label{Sect:Intro}
We consider the elliptic equation
\begin{equation} \label{main_problem}
\LL u = \sigma u^{q} + \mu \quad \text{in} \;\; \Omega,
\quad u=0 \;\; \text{on} \;\;  \partial\Omega,
\end{equation}
in the sublinear case $0<q<1$. Here $\Omega$ is an arbitrary domain 
(a nonempty open connected set) in $\R^n$, $n \geq 2$, which possesses a positive Green function, 
and $\sigma$, $\mu$ are nonnegative locally integrable functions, 
or more generally, nonnegative Radon measures in $\Omega$. 
This class of (locally finite)  measures is denoted by $\mathcal{M}^{+}(\Omega)$.

The operator $\LL u := -\,\text{div}\,(\A \nabla u)$ with bounded measurable coefficients is assumed to be uniformly elliptic, 
i.e., $\A: \Omega \rightarrow \R^{n \times n}$ is a real symmetric 
matrix-valued function on $\Omega$, and  
there exist positive constants $m\leq M$ such that 
\begin{equation}\label{cond:ellipticity}
m |\xi|^{2} \leq \A(x)\xi \cdot \xi \leq M |\xi|^{2},
\end{equation}
for almost every $x \in \Omega$
and every $\xi \in \R^n$.

Both homogeneous ($\mu \equiv 0$) and inhomogeneous equations  
($\mu \not\equiv 0$) 
will be studied simultaneously. The latter case involves a nontrivial 
question  regarding interaction between $\sigma$ and $\mu$  
(see Lemma \ref{lem_relation} below).

We denote by $\dot{W}^{1,p}_{0}(\Omega)$ ($1\le p<\infty$)  the homogeneous Sobolev (or Dirichlet) 
space defined as the closure of $C_{0}^{\infty}(\Omega)$ with respect to the seminorm 
$\| u \|_{\dot{W}^{1,p}_{0}(\Omega)} :=  \| \nabla u \|_{L^{p}(\Omega)}$.

An indispensable tool in our study is the notion of the generalized energy. 
Suppose $u$ is a positive Green potential, i.e., 
$u=\g\omega$ for $\omega \in \M^{+}(\Omega)$ with $\omega \not\equiv 0$, 
\[
\g\omega(x):=\int_{\Omega}G(x,y)\;d\omega(y), \quad x \in \Omega,
\] 
where $G$ is a positive Green function associated with $\LL$ (see \cite{K}).

As we will show below (see Theorem \ref{Thm:Energy}), 
the condition
\begin{equation}\label{gamma-norm}
\mathbb{E}_{\gamma}[u]:=\int_{\Omega} |\nabla u|^{2} u^{\gamma-1}\;dx < +\infty, \quad \gamma > 0
\end{equation}
is equivalent to the generalized Green energy $\E_{\gamma}[\omega]$ being finite,
\begin{equation}\label{generalized Green energy}
\E_{\gamma}[\omega] := \int_{\Omega} (\g\omega)^{\gamma}\;d\omega
= \int_{\Omega} u^{\gamma}\;d\omega < +\infty,
\end{equation}
which is also equivalent to $u^{\frac{\gamma +1}{2}} \in \dot{W}_{0}^{1,2}(\Omega)$. 
In this case, we  have  
\begin{equation}\label{IBP1}
\int_{\Omega} u^{\gamma}\;d\omega  = \gamma \int_{\Omega} ( \A\nabla u \cdot \nabla u) u^{\gamma-1}\;dx.
\end{equation}

This is well-known in the case 
$\gamma = 1$ for the Laplacian $\LL = -\Delta$, see \cite{L, MZ}. 
Analogous integration by parts formulas with $\gamma>0$ for functions $u$ in certain 
Sobolev spaces with various extra restrictions on $\Omega$ and $\omega$ can be found in \cite{MS}.

Two other key ingredients in our approach are weighted norm inequalities for 
Green potentials of the type $\g: L^r(\Omega, d\omega) \rightarrow 
L^s (\Omega, d\omega) $ for arbitrary $\omega \in \mathcal{M}^{+}(\Omega)$, 
in the non-classical case $0<s<r$ and $r>1$ (Theorem \ref{thm_Ver17}), along with
iterated pointwise estimates for Green potentials (Theorem \ref{Thm:iterated}) discussed below.  

Employing these tools, we establish \textit{necessary and sufficient} conditions 
on $\sigma$ and $\mu$, in terms of their generalized Green energy \eqref{generalized Green energy}, 
for the existence of a positive $\A$-superharmonic solution $u \in L^{q}_{loc}(\Omega, d\sigma)$ (see Definition \ref{def:sol}) 
 to  \eqref{main_problem} that satisfies  \eqref{gamma-norm}.

In the case $\gamma =1$, we also show  that such  a solution $u \in \dot{W}^{1,2}_{0}(\Omega)$
(the so-called finite energy solution) is unique. Notice that when $\LL = -\Delta$, 
both the existence and uniqueness of a finite energy 
solution to \eqref{main_problem} was obtained by the authors in \cite{SV} (see also \cite{CV} for $\Omega =\R^n$).

When $0<\gamma<1$, our result gives sufficient conditions
for the existence of a positive solution $u \in \dot{W}^{1,p}_{0}(\Omega)$ to 
\eqref{main_problem} where $1<p<2$.

Existence and uniqueness of \textit{bounded} solutions to \eqref{main_problem} on $\Omega =\R^n$ (with $\mu= 0$) was established by Brezis and Kamin in \cite{BK}. 

We state our main result and its consequences as follows.

\begin{Thm}\label{thm-main}
Let $0<q<1$ and  $\gamma>0$. Suppose  $G$ is a positive Green function associated with $\LL$ in $\Omega\subseteq \R^n$ $(n \ge 2)$. Let $\sigma, \mu \in \mathcal{M}^{+}(\Omega)$ so that 
$\sigma \not\equiv 0$.  Then there exists a positive solution $u \in L^{q}_{loc}(\Omega, d\sigma)$ to
\eqref{main_problem} which satisfies \eqref{gamma-norm}, or equivalently  $u^{\frac{\gamma +1}{2}} \in \dot{W}_{0}^{1,2}(\Omega)$, if and only if the following conditions hold:
\begin{equation}\label{cond_sigma}
\g\sigma \in L^{\frac{\gamma + q}{1-q}}(\Omega, d\sigma)
\end{equation}
and 
\begin{equation}\label{cond_mu}
\g\mu \in L^{\gamma}(\Omega, d\mu).
\end{equation}
When $\gamma = 1$, such a solution is unique in $\dot{W}_{0}^{1,2}(\Omega)$.
\end{Thm}

\begin{Cor}\label{cor-main}
Under the assumptions of Theorem \ref{thm-main},  suppose that 
  $\frac{n}{n-1}< p \le 2$, where $n \geq 3$. If both conditions \eqref{cond_sigma} and \eqref{cond_mu} hold with $\gamma := \frac{p(n-1)-n}{n-p} \in (0,1]$, i.e., 
\begin{equation}\label{cond:pair0}
\g\sigma \in L^{r}(\Omega, d\sigma)
\quad \text{and}  \quad 
\g\mu \in L^{s}(\Omega, d\mu),
\end{equation}
where $r:=\frac{p(n-2)}{(1-q)(n-p)} - 1$ and $s:= \frac{p(n-1)-n}{n-p}$,
then there exists a positive solution 
$u \in L^{q}_{loc}(\Omega, d\sigma) \cap \dot{W}_{0}^{1,p}(\Omega)$
to \eqref{main_problem}.
\end{Cor}

A sufficient condition for 
\eqref{cond:pair0}  is given by  
\begin{equation}\label{cond:pair1}
\sigma \in L^{r_{1}}(\Omega)
\quad \text{and} \quad 
\mu \in L^{s_{1}}(\Omega),
\end{equation}
where $ r_{1} := \frac{n(\gamma + 1)}{n(1-q)+2(\gamma +q)}$  
and  $s_{1} :=\frac{n(\gamma+1)}{n+2\gamma}$ (see Proposition \ref{Prop} below).
Thus, in light of Corollary \ref{cor-main}, we have the following result.

\begin{Cor}\label{cor-main2}
Under the assumptions of Corollary \ref{cor-main}, if 
\begin{equation}\label{cond:pair2}
\sigma \in L^{r_{2}}(\Omega)
\quad \text{and} \quad 
\mu \in L^{s_{2}}(\Omega),
\end{equation}
where $ r_{2} := \frac{np}{(n-p)(1-q)+2p}$  
and  $s_{2} :=\frac{np}{n+p}$, then there exists a positive solution 
$u \in L^{q}_{loc}(\Omega, d\sigma) \cap \dot{W}_{0}^{1,p}(\Omega)$ to \eqref{main_problem}.
\end{Cor}

We observe that Corollary \ref{cor-main2} in the  case $\mu \equiv 0$ (for  
bounded domains $\Omega$) is due to Boccardo and Orsina \cite{BO}, with a different proof.

We sketch our method of proof of Theorem \ref{thm-main} in 
the inhomogeneous case $\mu \not\equiv 0$. (The homogeneous case $\mu\equiv 0 $ is simpler since possible interaction between the nonlinear term involving $\sigma$ and $\mu$  is omitted.)

We start with  the corresponding integral equation
\begin{equation}\label{eq_int_0}
\tilde{u} = \g(\tilde{u}^{q}d\sigma) + \g\mu  \quad \text{in} \;\; \Omega, 
\end{equation}
under some mild assumptions on the kernel $G$, which by the maximum principle are automatically satisfied by Green functions associated with
elliptic operators (including $\LL$) in $\Omega$.

We find a crucial relation between $\sigma$ and $\mu$, which follows from  
conditions \eqref{cond_sigma} and \eqref{cond_mu}, and  yields an important two-weight condition:
\begin{equation}\label{cond_mu-sigma}
\g\mu \in L^{\gamma+q}(\Omega, d\sigma).
\end{equation}

This supplementary fact allows us to construct a positive solution
$\tilde{u} \in L^{\gamma+q}(\Omega, d\sigma) \cap L^{\gamma}(\Omega, d\mu)$
to \eqref{eq_int_0} by using an iterative procedure, under assumptions 
\eqref{cond_sigma} and \eqref{cond_mu}.

In this procedure, we employ the fact established  recently in \cite{V}
that condition \eqref{cond_sigma} is equivalent to the weighted norm inequality for Green potentials, 
\begin{equation}\label{weighted_norm_ineq_special0}
\big\| \g(f d\sigma) \big\|_{L^{\gamma +q}(\Omega,\;d\sigma)} 
\leq C \| f \|_{L^{\frac{\gamma+q}{q}}(\Omega,\;d\sigma)}, 
\quad \forall f \in L^{\frac{\gamma+q}{q}}(\Omega, d\sigma), 
\end{equation}
where $C$ is a positive constant independent of $f$. Therefore, either \eqref{cond_sigma} 
or \eqref{weighted_norm_ineq_special0}, together with \eqref{cond_mu}, turns out to be 
necessary and sufficient for the existence of a  positive solution  
 to \eqref{main_problem} satisfying \eqref{gamma-norm}.

When $G$ is a positive Green function associated with $\LL$ on $\Omega$,
the integral equation \eqref{eq_int_0} is equivalent to problem \eqref{main_problem}.
Appealing to our characterization of the generalized Green energy, 
\eqref{gamma-norm} $\Longleftrightarrow$ \eqref{generalized Green energy}
with $\omega:=\sigma u^{q} + \mu $, we deduce that \eqref{cond_sigma} and \eqref{cond_mu}
are necessary and sufficient for the existence of a positive solution $u \in L^{q}_{loc}(\Omega, d\sigma)$ 
to \eqref{main_problem} which satisfies \eqref{gamma-norm}.

This paper is organized as follows. 
In Sect. \ref{Sect:Prelim}, we recall some background facts in potential theory and PDE, and
collect useful results which are repeatedly referred to throughout this study.
Sect. \ref{Sect:Energy} is devoted to a characterization of the generalized Green energy. Our main result and its consequences are demonstrated in Sect. \ref{Sect:SolnEq}.

%%%%%%%%%%%%%%%%%%%%%%%%%%%%%%%%%%%%
%%%%%%%%%%%%%%%%%%%%%%%%%%%%%%%%%%%%
\section{Preliminaries}\label{Sect:Prelim}
Let $\Omega\subseteq \R^n$ be a domain, and  $\LL u := -\text{div}\,(\A \nabla u)$,
where $\A: \Omega \rightarrow \R^{n \times n}$  satisfies the uniform ellipticity condition \eqref{cond:ellipticity}.
%%%%%%%%%%%%%%%%%%%%%%%%%%%%%%%%%%%%
\subsection{Function spaces}
Denote by $C^{\infty}_{0}(\Omega)$ the set of all smooth compactly supported functions on 
$\Omega$.
For $0 < p<\infty$ and $\omega \in  \mathcal{M}^{+}(\Omega)$,
we denote by $L^{p}(\Omega, d\omega)$ the usual Lebesgue space of all real-valued measurable 
 functions $u$ on $\Omega$ such that 
\[
\Vert u \Vert_{L^{p}(\Omega,\,d\omega)} := \left( \int_{\Omega}|u|^{p}\;d\omega \right)^{\frac{1}{p}}
< +\infty.
\]
The corresponding local space is denoted by $L_{loc}^{p}(\Omega, d\omega)$.  

For $1 \leq p < \infty$, the Sobolev space $W^{1, p}(\Omega)$ consists of all functions 
$u \in L^p(\Omega)$ such that $|\nabla u | \in L^p(\Omega)$, where $\nabla u $ is the vector of 
distributional partial derivatives of $u$ of order 1, equipped with the norm 
\[
\| u \|_{W^{1, p}(\Omega)} := \| u \|_{L^{p}(\Omega)} + \| \nabla u \|_{L^{p}(\Omega)}.
\]
The corresponding local space is denoted by  $W_{loc}^{1, p}(\Omega)$. 
The Sobolev space $W_{0}^{1, p}(\Omega)$ is defined as the closure of $C_{0}^{\infty}(\Omega)$ 
in $W^{1, p}(\Omega)$. It is easy to see that $W_{0}^{1, p}(\R^n) = W^{1, p}(\R^n)$. 
The homogeneous version of  $W_{0}^{1, p}(\Omega)$, called the homogeneous Sobolev space 
(or Dirichlet space), denoted by $\dot{W}^{1,p}_{0}(\Omega)$, is defined as 
the closure of $C_{0}^{\infty}(\Omega)$ with respect to the seminorm 
\[
\| u \|_{\dot{W}^{1,p}_{0}(\Omega)} := \| \nabla u \|_{L^{p}(\Omega)}.
\]
That is, $\dot{W}^{1,p}_{0}(\Omega)$ is the set of all functions $u \in W^{1,p}_{loc}(\Omega)$ 
such that $|\nabla u| \in L^{p}(\Omega)$ for which there exists a sequence 
$\{ \varphi_{j}\}_{j=1}^{\infty} \subset C_{0}^{\infty}(\Omega)$ such that 
$\| \nabla u - \nabla \varphi_{j} \|_{L^{p}(\Omega)} \rightarrow 0 $ as $j \rightarrow \infty$.
When $1 < p < n$, the dual space to $\dot{W}_{0}^{1,p}(\Omega)$ denoted by 
$\dot{W}^{-1,p'}(\Omega)$, is the space of distributions $\omega\in \mathcal{D}'(\Omega)$ such that
\[
\| \omega \|_{\dot{W}^{-1,p'}(\Omega)} := 
\sup \frac{|\langle \omega, u \rangle|}{ \| u \|_{\dot{W}_{0}^{1,p}(\Omega)}} < +\infty
\]
where the supremum is taken over all nontrivial functions $u \in C^\infty_{0}(\Omega)$. 
Here $p':= \frac{p}{p-1}$ is the H\"{o}lder conjugate of $p$.
%%%%%%%%%%%%%%%%%%%%%%%%%%%%%%%%%%%%
\subsection{$\A$-superharmonic functions}
A function $u \in W_{loc}^{1,2}(\Omega)$ is said to be $\A$-harmonic if 
$u$ satisfies the equation
\begin{equation}\label{eq:laplace}
\LL u = 0 \qquad \text{in} \;\; \Omega
\end{equation}
in the distributional sense, i.e.,
\[
\int_{\Omega} \A\nabla u \cdot \nabla \varphi \;dx = 0,  
\quad \forall\varphi \in  C_{0}^{\infty}(\Omega).
\]
Every $\A$-harmonic function $u$ has a continuous 
representative which coincides with $u$ a.e. (see \cite[Theorem 3.70]{HKM}). 
We  denote by $\mathcal{H}_{\A}(\Omega)$ the set of all continuous 
$\A$-harmonic functions in $\Omega$.

A function $u: \Omega \rightarrow (-\infty, +\infty]$ is $\A$-superharmonic if 
$u$ is lower semicontinuous in $\Omega$, $u \not\equiv +\infty $ in each component of $\Omega$, 
and whenever $D$ is a relatively compact open subset of  $\Omega$ and
$h \in C(\overline{D}) \cap \mathcal{H}_{\A}(D)$, the inequality $h \leq u$ on $\partial D$ 
yields $h \leq u$ on $D$. A function $u$ in $\Omega$ is called $\A$-subharmonic 
if $-u$ is $\A$-superharmonic.

Every $\A$-superharmonic function $u$ in $\Omega$ is quasicontinuous in $\Omega$ 
\cite[Theorem 10.9]{HKM}, which means that for every $\epsilon > 0$, there is an open set 
$G \subset \Omega$ such that $\text{cap}(G)<\epsilon$ and the restriction 
$u_{| \Omega \setminus  G}$ is continuous on $\Omega \setminus  G$. 

Here the capacity of an open set $G\subset\Omega$ is defined by 
\[
\text{cap}(G):= \sup \lbrace \text{cap}(K): \, K \subset \Omega \;\; \text{compact} \rbrace, 
\]
where the capacity of a compact set $K \subset \Omega$ is given by
\[
\text{cap}(K) :=  \inf \big\lbrace \Vert  \nabla u \Vert^{2}_{L^{2}(\Omega)}\colon \, \,  u \geq 1 \;\; \text{on} \;\; K, \quad u \in C_{0}^{\infty}(\Omega) \big\rbrace.
\]
For an arbitrary set $E\subset \Omega$, 
\[
\text{cap}(E):= \inf \lbrace \text{cap}(G): \, E \subseteq G \subset \Omega, \;\; G \, \, \text{open} \rbrace. 
\]

A statement is said to hold quasi-everywhere (q.e.) in $\Omega$  if it holds everywhere 
except for a set of capacity zero in $\Omega$. 

Denote by $\mathcal{M}_0^{+}(\Omega)$ the class of all measures $\omega \in \M^{+}(\Omega)$
which are absolutely continuous with respect to capacity, that is, 
$\omega(K) = 0$ whenever $\text{cap}(K) = 0$ for every compact subset $K$ in $\Omega$.
It follows by Poincar\'{e}'s inequality \cite[Corollary 1.57]{MZ} that Lebesgue measure is
absolutely continuous with respect to the capacity.

Let $u$ be an $\A$-superharmonic function in $\Omega$. Then $u \in L^{r}_{loc}(\Omega)$ 
and $|\nabla u| \in L^{s}_{loc}(\Omega)$ whenever $0<r<\frac{n}{n-2}$ and $0<s <\frac{n}{n-1}$. 
In particular, $u \in W^{1,p}_{loc}(\Omega)$ whenever $1 \leq p < \frac{n}{n-1}$.
Moreover, there exists a unique measure 
$ \omega \in \mathcal{M}^{+}(\Omega)$ 
such  that 
\begin{equation}\label{eq:Riesz-meas}
\LL u = \omega \quad \text{in} \;\; \Omega
\end{equation}
in the distributional sense, i.e.,
\[
\int_{\Omega} \A\nabla u \cdot \nabla \varphi \;dx = \int_{\Omega} \varphi \;d\omega,  
\quad \forall\varphi \in  C_{0}^{\infty}(\Omega).
\]
The measure $\omega$ is called the Riesz measure associated with $u$, often 
denoted by $\omega[u]$ (see \cite[Theorem 7.46, Theorem 21.2]{HKM}). 

For a positive $\A$-superharmonic function $u$ in $\Omega$, we shall later use the fact that 
each truncation $u_{k}:= \min(u,k)$, where $k \in \N$, is a positive $\A$-superharmonic 
function of the class $L^{\infty}(\Omega) \cap \dot{W}^{1,2}_{loc}(\Omega)$, 
and its Riesz measure $\omega[u_k]$ is 
locally in the dual of $W^{1,2}(\Omega)$. Moreover, $\omega[u_k] \rightarrow \omega[u]$
weakly in $\Omega$ as $k \rightarrow \infty$, see \cite{HKM, KKT}.
%%%%%%%%%%%%%%%%%%%%%%%%%%%%%%%%%%%%
\subsection{Kernels and potentials}
Let $G: \Omega \times \Omega \rightarrow (0, \infty]$ be a positive lower semicontinuous kernel. 
For $\omega \in \mathcal{M}^{+}(\Omega)$, the potential of $\omega$ is defined by 
\[
\g \omega (x) := \int_{\Omega} G(x,y) \;d\omega (y), \quad x \in \Omega.
\]
Observe that $\g\omega(x)$ is lower semi-continuous on $\Omega \times \M^{+}(\Omega)$ if  $G(x,y)$ is lower semi-continuous on $\Omega \times \Omega$, see \cite{Br}.

A positive kernel $G$ on $\Omega \times \Omega$ is said to satisfy the weak maximum principle (WMP) with constant $h \geq 1$ if for any $\omega \in \mathcal{M}^{+}(\Omega)$,
\begin{equation}\label{WMP}
\sup \{ \g\omega(x) : x \in \text{supp}(\omega) \} \leq 1
\Longrightarrow
\sup \{ \g\omega(x) : x \in \Omega \} \leq h.
\end{equation}
Here we use the notation supp$(\omega)$ for the support of $\omega$.

When $h=1$ in \eqref{WMP}, the kernel $G$ is said to satisfy the strong maximum principle.  It holds for Green functions associated with the classical Laplacian $-\Delta$, or more 
generally the linear uniformly elliptic operator in divergence form $\LL$, as well as 
the fractional Laplacian $(-\Delta)^{\alpha}$ in the case $0<\alpha\leq 1$, in every domain 
$\Omega \subset \R^n$ which possesses a Green function. 

The WMP holds for Riesz kernels on $\R^n$ associated with $(-\Delta)^{\alpha}$ in the full range 
$0< \alpha < \frac{n}{2}$, and more generally for all radially nonincreasing kernels on $\R^n$, 
see \cite{AH}.

We say that a positive kernel $G$ on $\Omega \times \Omega$ is quasi-symmetric if there exists 
a constant $a\geq 1$ such that 
\begin{equation}\label{quasi-sym}
a^{-1}G(y,x) \leq G(x,y) \leq a G(y,x), \quad x,y \in \Omega.
\end{equation}
When $a=1$ in \eqref{quasi-sym}, the kernel $G$ is said to be symmetric. 
There are many kernels associated with elliptic operators that are quasi-symmetric and 
satisfy the WMP, see \cite{An}.

We summarize that the Green function $G$ associated with $\LL$ on $\Omega$
is a positive lower semicontinuous symmetric kernel, 
which satisfies the strong maximum principle \cite{K, LSW}.
Further, for $\omega \in \M^{+}(\Omega)$, the Green potential $\g\omega$ is 
either $\A$-superharmonic or identically $+\infty$, in each component of
$\Omega$, see \cite{GH}.
%%%%%%%%%%%%%%%%%%%%%%%%%%%%%%%%%%%%
\subsection{Some known results}
We shall need the following weak continuity result established in \cite{TW}.

\begin{Thm} [\cite{TW}] \label{weak_cont}
Suppose $\lbrace u_{j} \rbrace_{j=1}^{\infty}$ is a sequence of positive $\A$-superharmonic functions 
in $\Omega$ such that $u_j \rightarrow u$ a.e. as $j \rightarrow \infty$, where $u$ is an
$\A$-superharmonic function in $\Omega$. Then $\omega[u_{j}]$ converges weakly to $\omega[u]$, that is,
\[
\lim_{j \rightarrow \infty} \int_{\Omega} \varphi \; d\omega[u_j] = \int_{\Omega} \varphi \; 
d\omega [u],
\quad \forall \varphi \in C_{0}^{\infty}(\Omega).
\]
\end{Thm}

The following theorem provides pointwise estimates for supersolutions 
to sublinear elliptic equations, see \cite[Theorem 1.3]{GV}.
\begin{Thm}[\cite{GV}]\label{Thm:lowerbound}
Let $0<q<1$, $\omega \in \mathcal{M}^{+}(\Omega)$, and let
$G$ be a positive lower semicontinuous kernel on 
$\Omega \times \Omega$, which satisfies the WMP with constant $h \geq 1$.
If $u \in L^{q}_{loc}(\Omega, d\omega)$ is a positive solution to the
integral inequality
\begin{equation}\label{int_ineq}
u \geq \g(u^{q}d\omega) \quad \text{in} \;\; \Omega,
\end{equation}
then 
\begin{equation}\label{lowerbound}
u(x)  \geq (1-q)^{\frac{1}{1-q}} h^{\frac{-q}{1-q}} [ \g \omega (x) ]^{\frac{1}{1-q}}, \quad  \forall x \in \Omega.
\end{equation}
\end{Thm}

The following pair of  iterated pointwise inequalities plays an important role in this paper (see \cite[Lemma 2.5]{GV}). 
\begin{Thm}[\cite{GV}]\label{Thm:iterated}
Let $\omega \in \M^{+}(\Omega)$, and let $G$ be a positive lower semicontinuous kernel on 
$\Omega \times \Omega$, which satisfies the WMP with constant $h\geq1$. 
Then the following estimates hold: 
\begin{itemize}
	\item[(i)]If $s\geq1$, then   
	\begin{equation}\label{iterated1}
(\g\omega)^{s}(x) \leq s h^{s-1} \, \g \left( (\g\omega)^{s-1} d\omega \right)(x), \quad  \forall x \in \Omega. 
	\end{equation}
	\item[(ii)] If $0<s\leq1$, then
	\begin{equation}\label{iterated2}
(\g\omega)^{s}(x) \geq s h^{s-1} \, \g \left( (\g\omega)^{s-1} d\omega \right)(x), \quad  \forall x \in \Omega.
	\end{equation}
\end{itemize}
\end{Thm}

Our argument also relies on the following result established in \cite[Theorem 1.1]{V}, which explicitly characterizes 
$(p,r)$-weighted norm inequalities
\begin{equation}\label{weighted_norm_ineq}
\big\| \g(f d\omega) \big\|_{L^{r}(\Omega,\,d\omega)} \leq C \| f \|_{L^{p}(\Omega,\,d\omega)}, 
\quad \forall  f \in L^{p}(\Omega, d\omega),
\end{equation}
where $C$ is a positive constant independent of $f$, in the case $0< r < p$ and $1< p < \infty$, for 
arbitrary $\omega \in \mathcal{M}^{+}(\Omega)$, under certain assumptions on 
the kernel $G$.

\begin{Thm}[\cite{V}]\label{thm_Ver17}
Let $\omega \in \mathcal{M}^{+}(\Omega)$ with $\omega \not\equiv 0$, and let $G$ be a positive 
quasi-symmetric lower semicontinuous kernel on $\Omega \times \Omega$, which satisfies the WMP.
\begin{itemize}
\item[(i)] If $1 < p < \infty$ and $0 < r < p$, then the $(p,r)$-weighted norm inequality 
\eqref{weighted_norm_ineq} holds if and only if 
\begin{equation}\label{energy1}
\g\omega \in L^{\frac{pr}{p-r}}(\Omega, d\omega).
\end{equation}

\item[(ii)] If $0 < q < 1$ and $q<r<\infty$, then there exists a positive solution 
$u \in L^{r}(\Omega, d\omega)$ to the integral inequality \eqref{int_ineq}
if and only if the weighted norm inequality \eqref{weighted_norm_ineq} holds with 
$p = \frac{r}{q}$, that is, 
\begin{equation}\label{weighted_norm_ineq_special}
\big\| \g(f d\omega) \big\|_{L^{r}(\Omega,\;d\omega)} 
\leq C \| f \|_{L^{\frac{r}{q}}(\Omega,\;d\omega)}, 
\quad \forall f \in L^{\frac{r}{q}}(\Omega,d\omega),
\end{equation}
where $C$ is a positive constant independent of $f$; or equivalently,
\begin{equation}\label{energy2}
\g\omega \in L^{\frac{r}{1-q}}(\Omega, d\omega).
\end{equation}
\end{itemize}
\end{Thm}

The following  inequalities are often used in the 
theory of Schr\"{o}dinger operators and potential theory.
They can be found, for example, in \cite[Theorem 7.48]{HKM} and 
\cite[Proposition 1.5]{JMV}, respectively.

\begin{Thm}[\cite{HKM, JMV}]\label{ineq-a-b}
Let $\omega \in \M^{+}(\Omega)$ with $\omega \not\equiv 0$, 
and let $G$ be a positive Green function associated with $\LL$ on $\Omega$. 
Suppose $u:=\g\omega$ so that $u\not\equiv+\infty$. Then
there exists a positive constant $C$ which depends only on the ellipticity constants $m, M$ such that 
\begin{equation}\label{ineq-a}
\int_\Omega |\varphi|^2   \, \frac{|\nabla u|^2}{u^2} d x \le 
C \,  \int_\Omega |\nabla \varphi|^2  \;dx
\end{equation}
and
\begin{equation}\label{ineq-b}
 \int_\Omega |\varphi|^2   \, \frac{d \omega}{u}  \le 
C \int_\Omega |\nabla \varphi|^2  \;dx,
\end{equation}
for all (quasicontinuous representatives of) $\varphi \in \dot{W}_{0}^{1,2}(\Omega)$. 
\end{Thm}
%%%%%%%%%%%%%%%%%%%%%%%%%%%%%%%%%%%%
%%%%%%%%%%%%%%%%%%%%%%%%%%%%%%%%%%%%
\section{Generalized energy of measures}\label{Sect:Energy}
Let $\gamma >0$ and $\omega \in \M^{+}(\Omega)$, and let 
$G$ be a positive lower semicontinuous kernel on $\Omega \times \Omega$.
Define the $\gamma$-energy  
 of  $\omega$ by 
\begin{equation}\label{Gamma-E}
\E_{\gamma}[\omega] := \int_{\Omega} (\g\omega)^{\gamma} \; d\omega.
\end{equation}
In the case  $\gamma =1$, we use the notation $\E[\omega]:= \E_{1}[\omega]$.
Observe that $\E_{\gamma}[\omega]$ is well-defined, even though it may be infinite. 

When $G$ is a quasi-symmetric  kernel on 
$\Omega \times \Omega$ which satisfies the WMP, by
the definition of $\E_{\gamma}[\omega]$ and Theorem \ref{thm_Ver17} with 
$ r:= \frac{\gamma (1+\gamma)}{1+\gamma+\gamma^{2}} $
and $q:=\frac{\gamma^{2}}{1+\gamma+\gamma^{2}}$,
the following statements are equivalent:
\begin{itemize}
\item[(a)] $\E_{\gamma}[\omega] < + \infty $.
\item[(b)] $\g\omega \in L^{\gamma}(\Omega, d\omega)$.
\item[(c)] The weighted norm inequality \eqref{weighted_norm_ineq_special} is valid.
\item[(d)] There exists a positive solution $u \in L^{r}(\Omega, d\omega)$
to \eqref{int_ineq}.
\end{itemize}

There is also a similar characterization of $\E_{\gamma}[\omega]$, in terms of weak-type and 
strong-type inequalities, for nondegenerate kernels $G \ge 0$, see \cite{QV}.

We now consider 
 $\E_{\gamma}[\omega]$ in the case where $G$ is a positive Green function  
associated with $\LL$ on $\Omega$. 

\begin{Thm}\label{Thm:Energy}
Let $\gamma >0$ and $\omega \in \M^{+}(\Omega)$ with $\omega \not\equiv 0$,
and let $G$ be a positive Green function associated with $\LL$ on $\Omega$.
If $u:=\g\omega$ then the condition
\begin{equation}\label{gamma-energy-green}
\int_{\Omega} \left( \A\nabla u \cdot \nabla u \right) 
u^{\gamma -1}\;dx < +\infty
\end{equation}
is equivalent to $u^{\frac{\gamma +1}{2}} \in \dot{W}^{1,2}_{0}(\Omega)$ as well as
$(a)$, $(b)$, $(c)$ and $(d)$ above. In this case, we have
\begin{equation}\label{formula:Energy}
\E_{\gamma}[\omega] 
= \gamma \int_{\Omega} \left( \A\nabla u \cdot \nabla u \right) 
u^{\gamma -1}\;dx.
\end{equation}
\end{Thm}

\begin{Rem}
By uniform ellipticity condition \eqref{cond:ellipticity}, we see that
\eqref{gamma-energy-green} is equivalent to \eqref{gamma-norm}.
Therefore, by our discussion above, it suffices to show that 
\begin{equation}\label{whattoshow}
\E_{\gamma}[\omega]<+\infty 
\Longleftrightarrow
\mathbb{E}_{\gamma}[u]< +\infty
\Longleftrightarrow
u^{\frac{\gamma +1}{2}} \in \dot{W}^{1,2}_{0}(\Omega),
\end{equation}
and also establish formula \eqref{formula:Energy} whenever $\E_{\gamma}[\omega]<+\infty$.
\end{Rem}

We first prove an auxiliary fact which will be used in the 
proof of \eqref{whattoshow} when $0<\gamma <1$.

\begin{Lem}\label{lemma1} 
Let $0<\gamma<1$ and $\omega \in \M^{+}(\Omega)$, 
and let $G$ be a positive Green function associated with $-\LL$ on $\Omega$. 
Suppose $u:= \g\omega \not\equiv \infty$. 
Then $v:= u^{\gamma}$ is a positive $\A$-superharmonic function on $\Omega$, and 
$v=\g\mu$, where $\mu \in \M^{+}(\Omega)$ is the Riesz measure of $v$. Moreover, 
\begin{equation}\label{eq1:lemma1}
\E_\gamma[\omega]< +\infty
\Longleftrightarrow 
\E[\mu]< +\infty.
\end{equation}
In this case, we have
\begin{equation}\label{eq2:lemma1}
\frac{\gamma+1}{2} \, \E_\gamma[\omega]  
\leq  \E[\mu] \leq 
\frac{\gamma+1}{2 \gamma} \, \E_\gamma[\omega].
\end{equation}
\end{Lem}
\begin{proof} 
Notice that $u:=\g \omega$ is a positive $\A$-superharmonic function on $\Omega$ 
since $\g\omega \not\equiv +\infty$, see \cite{GH}.
Since $0<\gamma< 1$, the map $x \longmapsto x^{\gamma}$, for $x\geq 0$, is concave and 
increasing; it follows that $v:=u^{\gamma}$ is a positive $\A$-superharmonic function on $\Omega$ \cite[Theorem 7.5]{HKM}. In light of the Riesz decomposition theorem, 
$v=\g \mu+h$ where $\mu  \in \mathcal{M}^{+}(\Omega)$ is the Riesz measure of $v$,
and $h$ is the unique positive $\A$-harmonic function on $\Omega$.

Observe that $g:=h^{\frac{1}{\gamma}}$ is a positive $\A$-subharmonic function on $\Omega$
since $h$ is positive $\A$-harmonic and the map $x \longmapsto x^{\frac{1}{\gamma}}$ for $x \geq 0$,  is convex \cite[Theorem 7.5]{HKM}. Therefore $-g$ is an $\A$-superharmonic function on $\Omega$,
and thus  $-g= \g \nu + \tilde h$, where $\nu \in \mathcal{M}^{+}(\Omega)$ is the 
Riesz measure of $-g$, and $\tilde h$ is the unique $\A$-harmonic function on $\Omega$.
Since 
\[
\g\omega = u = v^{\frac{1}{\gamma}}=(\g \mu+h)^{\frac{1}{\gamma}} \geq h^{\frac{1}{\gamma}} = g = -\g \nu -\tilde h,
\] 
we deduce $\g (\omega + \nu) = u + \g \nu  \geq -\tilde h$. In other words,
$-\tilde h$ is a positive $\A$-harmonic minorant of the potential $\g (\omega + \nu)$. 
Consequently, $-\tilde h= 0$ and thus $g = -\g \nu \le 0$. 
This yields  $h=g^{\gamma}=0$. We have shown that  $v=\g \mu$ is a potential.

Suppose that $\mathcal{E}_\gamma[\omega]< +\infty$, and let $w:=u^{\frac{\gamma+1}{2}}$. 
Since $\frac{\gamma+1}{2}\in (0,1)$,
by a similar argument as above, $w$ is a positive $\A$-superharmonic function on $\Omega$, and 
$w = \g\mu$, where $\mu \in \M^{+}(\Omega)$ is the Riesz measure of $w$.
For each $k \in \N$, let 
$
u_k :=\min(u, k) \quad \text{and} \quad w_k :=\min(w, k^{\frac{\gamma+1}{2}}).  
$
Using the same argument as above, we see that both $u_k$ and $w_k$ are potentials with 
the corresponding Riesz measures  $\omega_k =\omega [u_k]$ and $\mu_k = \mu[w_k]$,
respectively. Clearly, $\text{supp}(\mu_k) \subset \{ u \le k\}$, thus both $u_k$ and 
$w_k$ are  uniformly bounded (by $k$)  $d \mu_k$-a.e.  Using Fubini's theorem 
 and the iterated inequality \eqref{iterated2} with 
 $\omega:=\mu_k$, $s:=\frac{2\gamma}{\gamma+1}$ and $h:=1$, we estimate 
\begin{align*}
\int_\Omega \g \mu_k \, d \mu  & =  \int_\Omega \g \mu \, d \mu_k = \int_\Omega \g \omega \,  (\g \mu)^ \frac{\gamma-1}{\gamma +1}  d \mu_k \leq \int_\Omega \g \omega \,  (\g \mu_k)^ \frac{\gamma-1}{\gamma +1}  d \mu_k \\
&= \int_\Omega \g \Big( (\g \mu_k)^{\frac{\gamma-1}{\gamma +1}} d \mu_k\Big) \,  d \omega \leq \frac{\gamma+1}{2\gamma}\int_\Omega (\g \mu_k)^{\frac{2 \gamma}{\gamma +1}} d \omega\\
&\leq \frac{\gamma+1}{2\gamma} \int_\Omega (\g \omega)^{\gamma} d \omega =  \frac{\gamma+1}{2\gamma} \mathcal{E}_\gamma[\omega].  
\end{align*}
Passing to the limit $k \rightarrow \infty$ and using the monotone convergence theorem, we deduce
\[
\mathcal{E}[\mu]=\int_\Omega \g \mu \, d \mu\le  \frac{\gamma+1}{2\gamma} 
 \mathcal{E}_\gamma[\omega] < +\infty
\]
since $w_{k}=\g \mu_k\uparrow w = \g \mu$ in $\Omega$.

Conversely, suppose  $\mathcal{E}[\mu]< +\infty$. By using the same notation  
and argument as above, we estimate 
\begin{align*}
\mathcal{E}[\mu]  &=\int_\Omega \g\mu \, d \mu 
=\int_\Omega (\g\omega)^{\frac{\gamma+1}{2}} \,d\mu\ge \int_\Omega (\g\omega_k)^{\frac{\gamma+1}{2}} \, d \mu \\ & 
 \ge \frac{\gamma+1}{2} \int_\Omega \g \Big ( (\g \omega_k)^{\frac{\gamma-1}{2}} d \omega_k \Big)  d\mu \ge \frac{\gamma+1}{2} \int_\Omega \g \Big ( (\g \omega)^{\frac{\gamma-1}{2}} d \omega_k \Big)  d \mu.
\end{align*}
In the above estimate, we have $u:=\g \omega \le k$ and $w=\g\mu \leq k$ $d \omega_k$-a.e.  
Applying  Fubini's theorem yields
\begin{align*}
\int_\Omega \g \Big ( (\g \omega)^{\frac{\gamma-1}{2}} d \omega_k \Big) \; d\mu
&= \int_\Omega (\g \omega)^{\frac{\gamma-1}{2}} \g \mu \, d \omega_k \\
&= \int_\Omega (\g \omega)^{\gamma} \, d \omega_k.
\end{align*}
Since $\gamma\in (0, 1)$, we have $ u^{\gamma}:=(\g \omega)^{\gamma}$ is a potential, that is,
$u^{\gamma}= \g \nu$, where $\nu \in \mathcal{M}^{+}(\Omega)$
is the Riesz measure of $u^{\gamma}$. Applying  Fubini's theorem and the monotone convergence theorem, we have
\[
\int_\Omega (\g \omega)^{\gamma} \, d \omega_k
=\int_\Omega \g \nu \, d \omega_k 
= \int_\Omega \g  \omega_k \, d \nu 
\uparrow \int_\Omega \g\omega \,d\nu
\]
since $\g  \omega_k=u_k\uparrow u=\g \omega$ in $\Omega$ as $k \rightarrow \infty$. 
Hence, 
\[
 \frac{\gamma+1}{2} \mathcal{E}_\gamma[\omega] 
 = \frac{\gamma+1}{2}  \int_\Omega \g  \nu \;d \omega
 = \frac{\gamma+1}{2}  \int_\Omega \g  \omega \;d \nu
\leq \mathcal{E}[\mu]
<+\infty.
\]
This completes the proof of the lemma.
\end{proof}

We now establish \eqref{whattoshow}, which yields the first part of Theorem \ref{Thm:Energy}.

\begin{Lem}\label{lemma2} 
Let $\gamma >0$ and  $\omega \in \M^{+}(\Omega)$ with $\omega \not\equiv 0$, 
and let $G$ be a positive Green function associated with $\LL$ on $\Omega$. 
If $u:=\g\omega$ then \eqref{whattoshow} holds.
In this case, there exists a positive constant $C$ which depends only on $m$, $M$ and $\gamma$ such that 
\begin{equation}\label{eq-equiv}
C^{-1} \mathcal{E}_\gamma[\omega] 
\leq \,
\mathbb{E}_{\gamma}[u]
\leq C \, 
\mathcal{E}_\gamma[\omega].
\end{equation}
\end{Lem}

\begin{proof} 
Without loss of generality, we may assume that $u \not\equiv+\infty$. 
It follows that $u$ is a positive $\A$-superharmonic 
function in $\Omega$. Moreover, $u\in W^{1,p}_{{\rm loc}} (\Omega)$
whenever $1\leq p <\frac{n}{n-1}$, see \cite{HKM}. 
Consider three cases as follows:

$\bullet$ Case $\gamma =1$. This is completely analogous to the classical result shown in,
for example, \cite[Theorem 1.20]{L}, due to uniform ellipticity assumption \eqref{cond:ellipticity}. 
In this case we further have formula \eqref{formula:Energy} using 
an approximation argument demonstrated below in Lemma \ref{Lem_identity}.

$\bullet$ Case $0< \gamma < 1$. In light of Lemma \ref{lemma1},  
we have $v:=u^{\frac{\gamma+1}{2}}$ is a positive $\A$-superharmonic function in $\Omega$, and 
$v= \g \mu$ where $\mu \in \M^{+}(\Omega)$ is the Riesz measure of $v$. 
Moreover, 
\[
\mathcal{E}_\gamma[\omega]< +\infty 
\Longleftrightarrow
\mathcal{E}[\mu]< +\infty
\Longleftrightarrow 
v \in \dot{W}^{1,2}_{0}(\Omega).
\]
In this case, we have
\[
\frac{\gamma+1}{2} \, \E_\gamma[\omega]  
\leq  \E[\mu] \leq 
\frac{\gamma+1}{2 \gamma} \, \E_\gamma[\omega].
\]
On the other hand, notice that $\nabla v=\frac{\gamma+1}{2}u^{\frac{\gamma-1}{2}} \nabla u$ a.e. in $\Omega$. Appealing to the previous case, we deduce 
\[
\mathcal{E}[\mu]= 
\int_\Omega \A \nabla v \cdot \nabla v\;dx 
\leq M \Big(\frac{\gamma+1}{2} \Big)^2
 \int_\Omega |\nabla u|^2 u^{\gamma-1}\;dx
\]
and similarly
\[
\mathcal{E}[\mu]= 
\int_\Omega \A \nabla v \cdot \nabla v\;dx 
\geq m \Big(\frac{\gamma+1}{2} \Big)^2
 \int_\Omega |\nabla u|^2 u^{\gamma-1}\;dx.
\]
This proves both assertions \eqref{whattoshow} and \eqref{eq-equiv}, respectively.

$\bullet$ Case $\gamma >1$. Suppose that  $v := u^{\frac{\gamma+1}{2}} \in \dot{W}^{1,2}_0(\Omega)$. Therefore,  $\mathbb{E}_{\gamma}[u] < + \infty$.
For each $k \in \N$, let $u_k=\min(u, k)$, which is a positive $\A$-superharmonic function of the class $L^{\infty}(\Omega) \cap W_{loc}^{1,2}(\Omega)$. 
Denote the corresponding Riesz measure of each $u_{k}$ 
by $\omega_k:= \omega[u_{k}]$. 
Without loss of generality, we may suppose $v$ is quasicontinuous.
Applying inequality \eqref{ineq-b} with $\phi:=v$ and $\omega:=\omega_k$, we obtain 
\[
\begin{split}
 \int_\Omega u^\gamma d \omega_k  
 &\leq \int_\Omega v^2 \frac{d \omega_k}{\g \omega_k}  \leq C \int_{\Omega} |\nabla v|^{2}\;dx \\ 
 & = C \Big(\frac{\gamma+1}{2}\Big)^2 \int_\Omega |\nabla u |^2 u^{\gamma-1} dx.
\end{split}
\]
 On the other hand, by Fubini's theorem and iterated estimate \eqref{iterated1} 
 with $\omega:=\omega_k$ and $s:=\gamma$, we have 
\[
\begin{split}
 \int_\Omega u^\gamma \;d\omega_k  
 &= \int_\Omega (\g \omega)^{\gamma-1} \g\omega \; d \omega_k  = \int_\Omega \g \Big( (\g \omega)^{\gamma-1} d \omega_k  \Big)\;d\omega \\  
 &\geq  \int_\Omega \g \Big( (\g \omega_k)^{\gamma-1} d \omega_k \Big) \;d\omega \geq  \frac{1}{\gamma} \int_\Omega (\g \omega_k)^{\gamma} d \omega. 
\end{split}
\]
Therefore,
\[
\frac{1}{\gamma} \int_\Omega (\g \omega_k)^{\gamma}\;d\omega
\leq 
C \Big(\frac{\gamma+1}{2}\Big)^2 \int_\Omega |\nabla u|^2 u^{\gamma-1}\;dx.
\]
Passing to the limit $k\to \infty$ and using the monotone convergence theorem, 
we obtain
 \[
\E_{\gamma}[\omega] = \int_\Omega u^{\gamma}\;d\omega 
 \leq C \gamma \Big(\frac{\gamma+1}{2}\Big)^2 \int_\Omega |\nabla u|^2 u^{\gamma-1} dx < +\infty
\]
since $u_k=\g \omega_k\uparrow u=\g \omega$ in $\Omega$.

Conversely, suppose $\E_\gamma[\omega]< +\infty$.
Write $\gamma=\frac{1+q}{1-q}$ where $0<q<1$,
and consider the corresponding sublinear elliptic equation  
 \begin{equation}\label{sublin}
\LL w = \omega \, w^q \quad \text{in} \; \; \Omega.
 \end{equation}
Using a similar argument as in the proof of \cite[Lemma 5.5]{SV}, there exists a positive 
finite energy solution $w \in \dot{W}^{1,2}_0(\Omega)$ to \eqref{sublin}, satisfying
\begin{equation}\label{est-w}
 ||w||_{\dot{W}^{1,2}_0(\Omega)} \le c \,  \Big(\int_\Omega (\g \omega)^{\frac{1+q}{1-q}}
 \;d\omega \Big)^{\frac{1}{2}},
 \end{equation}
where $c=c(m,M,q)>0$. As usual, we may assume $w$ is quasicontinuous.
By Theorem \ref{Thm:lowerbound}, $w$ obeys the lower bound
\[
 w \ge (1-q)^{\frac{1}{1-q}} (\g \omega)^{\frac{1}{1-q}}  \quad \text{in} \, \, \Omega.
\]
Therefore,
\begin{equation}\label{lower}
 v := u^{\frac{\gamma+1}{2}} = (\g \omega)^{\frac{\gamma+1}{2}}\le (1-q)^{-\frac{1}{1-q}} w  \quad \text{in} \, \, \Omega.
 \end{equation}
From this, we deduce
\begin{equation}\label{est:gamma-norm}
\begin{split}
\int_\Omega |\nabla u|^2 u^{\gamma-1}\;dx 
&= \int_\Omega v^2 \frac{|\nabla u|^2}{u^{2}}\;dx \\
&\leq (1-q)^{-\frac{2}{1-q}}  \int_\Omega w^2 \frac{|\nabla u|^2}{u^{2}}\;dx.
\end{split}
 \end{equation}
Using inequality \eqref{ineq-a} with $\phi:=w$, along with 
 \eqref{est-w}, we estimate 
\begin{equation}\label{est:w2}
\begin{split}
  \int_\Omega w^2 \frac{|\nabla u|^2}{u^{2}}\;dx 
  \leq C \, ||w||^2_{\dot{W}^{1,2}(\Omega)}
  \leq C c^2 \, \mathcal{E}_\gamma[\omega]. 
 \end{split}
 \end{equation}
Hence, by \eqref{est:gamma-norm} and \eqref{est:w2}, we arrive at
\[
\int_\Omega |\nabla u|^2 u^{\gamma-1}\;dx 
\leq Cc^{2} (1-q)^{-\frac{2}{1-q}} \, \E_\gamma[\omega] 
< +\infty.
\]
Moreover, for $r=\frac{2n}{n-2}$ if $n \ge 3$, and any $r<\infty$ if $n=2$, 
we have that $v \in L^r(\Omega)$, since the same is true for $w\in \dot{W}^{1,2}_0(\Omega)$. 
Recall that $\Omega$ is assumed to be a Green domain in the case $n=2$. In other words, 
$v \in \dot{W}^{1, 2}(\Omega)$, the corresponding Sobolev space equipped with the norm 
$ ||v||_{\dot{W}^{1, 2}(\Omega)} = ||\nabla v||_{L^2(\Omega)} + ||v||_{L^r(\Omega)}$.  
The fact that $v  \in \dot{W}^{1, 2}_0(\Omega)$ follows from the Deny-Lions theorem 
(see \cite[Sec. 9.12]{AH} and the references cited there). Notice that, for $z\in \partial \Omega$,  the quasi-limit  $\lim_{x \to z} v(x)=0$ q.e. by \eqref{lower}, since the same is true for   $w \in \dot{W}^{1,2}_{0}(\Omega)$ by \cite{Kol}, Corollary to Theorem 1. This finishes the proof of  lemma.
\end{proof}

We now complete the proof of Theorem \ref{Thm:Energy} by establishing formula \eqref{formula:Energy}
whenever $\E_{\gamma}[\omega]<+\infty$, using an approximation procedure.
  
\begin{Lem}\label{Lem_identity}
Let $\gamma>0$ and $\omega \in \mathcal{M}^{+}(\Omega)$ 
with $\omega \not\equiv 0$, and let $G$ be a positive Green function associated with 
$\LL$ on $\Omega$. If $u:=\g\omega$ then formula \eqref{formula:Energy} 
is valid whenever $\E_{\gamma}[\omega]<+\infty$.
\end{Lem}
\begin{proof}
Suppose $\E_{\gamma}[\omega]<+\infty$. Therefore both \eqref{gamma-energy-green}
and \eqref{gamma-norm} holds, in views of Lemma \ref{lemma2} and 
uniform ellipticity condition \eqref{cond:ellipticity}.

For each $k \in \N$, we set $u_{k}=\min(u,k)$. Notice that 
$u_{k}$ is a positive $\A$-superharmonic function of the class 
$L^{\infty}(\Omega) \cap W_{loc}^{1,2}(\Omega)$. 
Denote the corresponding Riesz measure of  $u_{k}$ by $\omega_k:= \omega[u_{k}]$. 

Let $\lbrace u_{k}^{(j)} \rbrace_{j \geq k}$ be a sequence of mollified $u_{k}$, 
defined on  $\Omega_{j}:=\lbrace x \in \Omega : \text{dist}(x,\partial\Omega)>\frac{1}{j}\rbrace$.
Denote $\omega_{k}^{(j)}:= \LL u_{k}^{(j)}$.
In addition, for $j \geq k$, select $\varphi_{j}
\in  C_{0}^{\infty}(\Omega)$ so that
\[
\begin{split}
0 \leq\varphi_{j} \leq 1, \quad
& \text{supp}\,\varphi_{j} \subset \Omega_{j}, \quad 
\varphi_{j} \uparrow \chi_{\Omega} \;\; \text{as} \;\; j \rightarrow \infty,  \\ 
& \int_{\Omega} | \nabla \varphi_{j} |^{2}\;dx \leq \frac{1}{j^{2\gamma+2}}.
\end{split}
\]
Using integration by parts, we obtain 
\[
\begin{split}
\gamma & \int_{\Omega} (\A\nabla u_{k}^{(j)} \cdot \nabla u_{k}^{(j)}) (u_{k}^{(j)})^{\gamma -1} \varphi_{j}\;dx
= \int_{\Omega} \nabla \left( (u_{k}^{(j)})^{\gamma} \right)  \cdot \A\nabla u_{k}^{(j)} 
\varphi_{j}\;dx \\
&= \int_{\Omega} (u_{k}^{(j)})^{\gamma} \varphi_{j} \;d\omega_{k}^{(j)}  - \int_{\Omega} (u_{k}^{(j)})^{\gamma} (\A\nabla u_{k}^{(j)} \cdot \nabla \varphi_{j})\;dx.
\end{split}
\]
Letting first $j \rightarrow \infty$, and then $k \rightarrow \infty$, we see that
\[
\gamma \int_{\Omega} |\nabla u_{k}^{(j)}|^{2} (u_{k}^{(j)})^{\gamma -1} \varphi_{j}\;dx 
\longrightarrow \gamma \int_{\Omega} |\nabla u|^{2} u^{\gamma -1}\;dx
\]
and
\[
\int_{\Omega} (u_{k}^{(j)})^{\gamma} \varphi_{j} \;d\omega_{k}^{(j)}
\longrightarrow \int_{\Omega} u^{\gamma}\;d\omega
\]
by means of mollification, the Lebesgue dominated convergence theorem, 
and weak continuity of $\LL$ (Theorem \ref{weak_cont}).
Moreover, by Schwarz's inequality, the construction of $\varphi_{j}$,
and uniform ellipticity condition \eqref{cond:ellipticity}, we deduce
\[
\begin{split}
& \left| \int_{\Omega} (u_{k}^{(j)})^{\gamma} (\A\nabla u_{k}^{(j)} \cdot \nabla \varphi_{j})\;dx \right|
\leq M^{\frac{1}{2}}\left( \int_{\Omega} |\nabla \varphi_{j}|^{2}\;dx \right)^{\frac{1}{2}} \\ & \times 
\left( \int_{\Omega} |\nabla u_{k}^{(j)}|^{2} (u_{k}^{(j)})^{2\gamma}\;dx \right)^{\frac{1}{2}} \\ & \leq \frac{M^{\frac{1}{2}}}{j^{{\gamma+1}}}
\left( \int_{\Omega} |\nabla u_{k}^{(j)}|^{2} (u_{k}^{(j)})^{\gamma -1} (u_{k}^{(j)})^{\gamma +1}
\;dx \right)^{\frac{1}{2}} \\
&\leq \frac{M^{\frac{1}{2}} k^{\frac{\gamma+1}{2}}}{k^{\gamma+1}}
\left( \int_{\Omega} |\nabla u|^{2} u^{\gamma -1} \; dx \right)^{\frac{1}{2}} =\frac{M^{\frac{1}{2}}}{k^{\frac{\gamma+1}{2}}}
\left( \int_{\Omega} |\nabla u|^{2} u^{\gamma -1} \; dx \right)^{\frac{1}{2}},
\end{split}
\]
which converges to zero as $k \rightarrow \infty$. This proves \eqref{formula:Energy}.
\end{proof}

The following lemma shows, in particular,  that if 
$\E_{\gamma}[\omega]<+\infty$ for some $\gamma > 0$, then  $\omega \in \M^{+}_{0}(\Omega)$.

\begin{Lem}\label{lemma-cap}
Let $\gamma>0$ and $\omega \in \mathcal{M}^{+}(\Omega)$, 
and let $G$ be a positive Green function associated with $\LL$ on $\Omega$.
Suppose that $u:=\g\omega \in L^{\gamma}_{loc}(\Omega, d\mu)$. 
Then for every compact set 
$K \subset \Omega$,
\begin{equation}\label{cap-mu}
\omega(K) 
\leq [{\rm cap}(K)]^{\frac{\gamma}{1+\gamma}} \left( \int_{K} u^{\gamma} \;d\omega \right)^{\frac{1}{1+\gamma}}.
\end{equation}
In particular, $\omega \in \M^{+}_{0}(\Omega)$.
\end{Lem}
\begin{proof}
Let $K$ be a compact subset of $\Omega$. By  \eqref{ineq-b}, we have 
\begin{equation}\label{est1-mu}
\int_{K} \frac{d\omega}{u} \leq \text{cap}(K).
\end{equation}
On the other hand, by H\"{o}lder's inequality, 
\begin{equation}\label{est2-mu}
\omega(K) 
= \int_{K} u^{\frac{-\gamma}{1+\gamma}} u^{\frac{\gamma}{1+\gamma}}\;d\omega
\leq \left( \int_{K} u^{-1} \;d\omega \right)^{\frac{\gamma}{1+\gamma}}
       \left( \int_{K} u^{\gamma} \;d\omega \right)^{\frac{1}{1+\gamma}}.
\end{equation}
Thus, \eqref{cap-mu} follows from \eqref{est1-mu} and \eqref{est2-mu}.
\end{proof}
%%%%%%%%%%%%%%%%%%%%%%%%%%%%%%%%%%%%
%%%%%%%%%%%%%%%%%%%%%%%%%%%%%%%%%%%%
\section{Positive solutions to sublinear elliptic equations}\label{Sect:SolnEq}
In this section, we prove our main result stated in Theorem \ref{thm-main} 
using the argument outlined earlier in Sec. \ref{Sect:Intro}. Its consequences are discussed 
here as well.

\begin{Def}\label{def:sol}
Let $q>0$ and $\sigma, \mu \in \mathcal{M}^{+}(\Omega)$. Let
$G$ be a positive Green function associated with $\LL$ on $\Omega$.
A solution $u$ to equation \eqref{main_problem} is understood in the sense that
$u$ is an $\A$-superharmonic function on $\Omega$ such that $u \in L^{q}_{loc}(\Omega, d\sigma)$ with $u \geq 0$ $d\sigma$-a.e., and
\begin{equation}\label{eq_int}
u = \g(u^{q}d\sigma) + \g\mu  \quad \text{in} \;\; \Omega.
\end{equation}
If further $u \in \dot{W}^{1,2}_{0}(\Omega)$, it is called a finite energy solution to \eqref{main_problem}.
\end{Def}

Our first theorem gives necessary and sufficient conditions for the existence of 
a positive solution $u \in L^{\gamma + q}(\Omega, d\sigma)$ ($\gamma>0$) to the
integral equation \eqref{eq_int} in the sublinear case $0<q<1$, 
under some mild assumptions on kernel $G$ satisfied by the Green function associated with
$\LL$ on $\Omega$.

\begin{Thm}\label{thm_int}
Let $0<q<1$, $\gamma >0$ and $\sigma, \mu \in \mathcal{M}^{+}(\Omega)$ 
with $\sigma \not\equiv 0$.
Suppose $G$ is a positive quasi-symmetric lower semicontinuous kernel on 
$\Omega \times \Omega$, which satisfies the WMP. If
\eqref{cond_sigma} and  \eqref{cond_mu-sigma} hold,
then there exists a positive solution $u \in L^{\gamma + q}(\Omega, d\sigma)$ to \eqref{eq_int}. The converse statement is valid without the quasi-symmetry assumption on $G$.
\end{Thm}

\begin{proof}
Suppose \eqref{cond_sigma} and \eqref{cond_mu-sigma} hold.
In the homogeneous case $\mu \equiv 0$, we can construct a monotone increasing 
sequence of positive functions 
$\lbrace  u_j \rbrace_{j=0}^{\infty} \subset L^{\gamma +q}(\Omega, d\sigma)$ by setting 
\[
u_0 := \kappa \left( \g\sigma \right)^{\frac{1}{1-q}} \quad \text{and} 
\quad u_{j+1} := \g (u^{q}_{j} d\sigma), \quad \text{for} \;\;  j \in \mathbb{N}_{0},
\]
where $\kappa>0$ is chosen to be sufficiently small. Then its pointwise limit 
$u:=\lim_{j \rightarrow \infty} u_j$ is a positive solution of the class
$L^{\gamma +  q}(\Omega, d\sigma)$ to \eqref{eq_int}, by the monotone convergence theorem  
(see \cite[Theorem 1.1\,(ii)]{V} for more details).

In the inhomogeneous case  $\mu \not\equiv 0$, we set  
\[
u_0 := \g\mu,  \qquad u_{j+1} := \g (u^{q}_{j} d\sigma) + \g\mu, \quad 
\text{for} \;\; j \in \mathbb{N}_{0}.
\]
Observe that $u_0 > 0$ since $\mu \not\equiv 0$, and
\[
u_{1} = \g (u^{q}_{0} d\sigma) + u_0 \geq u_0.
\]
Suppose $u_0 \leq u_1 \leq \ldots \leq u_j$ for some $j \in \N$. Then
\[
u_{j+1} 
= \g (u^{q}_{j} d\sigma) + \g\mu 
\geq \g (u^{q}_{j-1} d\sigma) + \g\mu
= u_{j}.
\]
Hence, by induction, $\lbrace  u_j\rbrace_{j=0}^{\infty}$ is an increasing sequence of positive functions. 
Further, each $u_j \in  L^{\gamma + q}(\Omega, d\sigma)$. Notice that 
\[
u_0 = \g \mu \in L^{\gamma + q}(\Omega, d\sigma),
\]
by the assumption \eqref{cond_mu-sigma}.
Suppose $u_0, \ldots, u_j \in L^{\gamma +q}(\Omega, d\sigma)$ for some $j \in \N$. Observe that
\begin{equation}\label{est1_thm_int}
\begin{split}
\Vert u_{j+1} \Vert_{L^{\gamma + q}(\Omega,\; d\sigma)}
&= \big\Vert \g (u^{q}_{j} d\sigma) + \g \mu \big\Vert_{L^{\gamma + q}(\Omega, \; d\sigma)} \\
&\leq c \big\Vert \g (u^{q}_{j} d\sigma) \big\Vert_{L^{\gamma + q}(\Omega, \; d\sigma)} 
    + c \big\Vert \g\mu  \big\Vert_{L^{\gamma + q}(\Omega, \; d\sigma)},
\end{split}
\end{equation}
where $c = \text{max} ( 1, 2^{\frac{1-\gamma-q}{\gamma + q}} )$. 
In view of Theorem \ref{thm_Ver17}, the assumption \eqref{cond_sigma} is equivalent to the 
weighted norm inequality \eqref{weighted_norm_ineq_special} with 
$\omega=\sigma$ and $r=\gamma+q$.
Therefore, we can estimate the first term on the right-hand side of \eqref{est1_thm_int} by applying
 \eqref{weighted_norm_ineq_special} with 
$f=u^{q}_{j} \in L^{\frac{\gamma + q}{q}}(\Omega, d\sigma)$,
\begin{equation}\label{est2_thm_int}
\begin{split}
& \big\Vert \g (u^{q}_{j} d\sigma) \big\Vert_{L^{\gamma + q}(\Omega, \;d\sigma)}
\leq C \left( \int_{\Omega} u^{\gamma + q}_{j}\;d\sigma \right)^{\frac{q}{\gamma + q}} \\
& \leq C \left( \int_{\Omega} u^{\gamma + q}_{j+1}\;d\sigma \right)^{\frac{q}{\gamma + q}} = C \Vert u_{j+1}\Vert^{q}_{L^{\gamma + q}(\Omega, \;d\sigma)}.
\end{split}
\end{equation}
By \eqref{est1_thm_int} and \eqref{est2_thm_int}, we have
\begin{equation} \label{est3_thm_int}
\Vert u_{j+1} \Vert_{L^{\gamma + q}(\Omega, \; d\sigma)}
\leq Cc \Vert u_{j+1}\Vert^{q}_{L^{\gamma + q}(\Omega, \;d\sigma)} 
+ c \big\Vert \g \mu  \big\Vert_{L^{\gamma + q}(\Omega, \;d\sigma)}.
\end{equation}
We estimate the first term on the right-hand side of \eqref{est3_thm_int} using Young's inequality,
\begin{equation}\label{est4_thm_int}
Cc \Vert u_{j+1}\Vert^{q}_{L^{\gamma + q}(\Omega, \;d\sigma)}
\leq q \Vert u_{j+1}\Vert_{L^{\gamma + q}(\Omega, \;d\sigma)} +  (1-q)(Cc)^{\frac{1}{1-q}}.
\end{equation}
Hence, by \eqref{est3_thm_int} and \eqref{est4_thm_int}, we obtain
\begin{equation}\label{est5_thm_int}
\Vert u_{j+1} \Vert_{L^{\gamma + q}(\Omega, \; d\sigma)} 
\leq (Cc)^{\frac{1}{1-q}} + \frac{c}{1-q} \big\Vert \g\mu
\big\Vert_{L^{\gamma + q}(\Omega, \; d\sigma)} < +\infty.
\end{equation}
By induction, we have shown that each $u_j \in  L^{\gamma + q}(\Omega, d\sigma)$. 
Finally, applying the monotone convergence theorem to the sequence 
$\lbrace u_{j} \rbrace_{j=0}^{\infty}$, we see that the pointwise limit $u:=\lim_{j \rightarrow \infty} u_j$
exists so that $u>0$, $u \in L^{\gamma + q}(\Omega, d\sigma)$, and satisfies \eqref{eq_int}.

Conversely, if there exists a positive solution $u \in L^{\gamma +  q}(\Omega, d\sigma)$ 
to \eqref{eq_int}, it is clear that \eqref{cond_mu-sigma} holds. Moreover, 
\eqref{cond_sigma} follows from the global pointwise lower bound \eqref{lowerbound}:
\[
u(x) \geq c [ \g \sigma (x) ]^{\frac{1}{1-q}}, \quad  \forall x \in \Omega,
\]
which does not require quasi-symmetry of $G$ (see \cite{GV}).
\end{proof}

An essential link between conditions \eqref{cond_sigma}, 
\eqref{cond_mu} and \eqref{cond_mu-sigma}, will be obtained in the next lemma.
It is an extension of \cite[Lemma 5.3]{SV} in the case $\gamma=1$; moreover,    $G$ does not need to be quasi-symmetric 
due to Theorem \ref{Thm:iterated}.

\begin{Lem}\label{lem_relation}
Let $0<q<1$, $\gamma >0$, and $\sigma, \mu \in \mathcal{M}^{+}(\Omega)$.
Suppose $G$ is a positive lower semicontinuous kernel
on $\Omega \times \Omega$, which satisfies the WMP.
Then conditions \eqref{cond_sigma} and \eqref{cond_mu}
imply \eqref{cond_mu-sigma}.
\end{Lem}

\begin{proof}
Without loss of generality, we may assume $\sigma, \mu \not\equiv 0$. 
Consider the following three cases:

$\bullet$ Case 1: $\gamma + q > 1$. Applying the iterated inequality
\eqref{iterated1} with $\omega= \mu$ and $s=\gamma +q$, together with 
Fubini's theorem and H\"{o}lder's inequality with the exponents
$\frac{\gamma}{\gamma+q-1}$ and $\frac{\gamma}{1-q}$, we obtain
\begin{equation}\label{est2_lem_relation}
\begin{split}
&\int_{\Omega} (\g \mu)^{\gamma+q} \;d\sigma 
\leq c \int_{\Omega} \g \left( (\g\mu)^{\gamma+q-1}\;d\mu\right) \;d\sigma \\
&= c \int_{\Omega} (\g \mu)^{\gamma+q-1} \g\sigma \;d\mu 
\\&\leq c \left[ \int_{\Omega} (\g \mu)^{\gamma}\;d\mu \right]^{\frac{\gamma+q-1}{\gamma}}
\left[ \int_{\Omega} (\g\sigma)^{\frac{\gamma}{1-q}} \;d\mu \right]^{\frac{1-q}{\gamma}}.
\end{split}
\end{equation}
The second integral on the right-hand side of \eqref{est2_lem_relation} is estimated by 
a similar argument as above.  
In fact, applying \eqref{iterated1} again with 
$\omega= \sigma$ and $s =\frac{\gamma}{1-q}$, along with Fubini's theorem 
and H\"{o}lder's inequality with the exponents
$ \frac{\gamma+q}{\gamma+q-1}$  and $\gamma + q$, we deduce
\begin{equation}\label{est3_lem_relation}
\begin{split}
\int_{\Omega} (\g \sigma)^{\frac{\gamma}{1-q}}\;d\mu
&\leq c \int_{\Omega} \g \left((\g\sigma)^{\frac{\gamma}{1-q}-1} d\sigma\right) \;d\mu \\
&= c \int_{\Omega} (\g \sigma)^{\frac{\gamma+q-1}{1-q}} \,\g\mu \;d\sigma \\
&\leq c \left[ \int_{\Omega} (\g \sigma)^{\frac{\gamma+q}{1-q}} \;d\sigma \right]
^{\frac{\gamma+q-1}{\gamma+q}}
\left[ \int_{\Omega} (\g \mu)^{\gamma+q} \;d\sigma \right]^{\frac{1}{\gamma+q}}.
\end{split}
\end{equation}
By \eqref{est2_lem_relation} and \eqref{est3_lem_relation}, we have 
\begin{equation}\label{est4_lem_relation}
\begin{split}
 \left[ \int_{\Omega} (\g \mu)^{\gamma+q}\;d\sigma \right]^{1 - \frac{1-q}{\gamma(\gamma+q)}}
& \leq c \left[ \int_{\Omega} (\g \mu)^{\gamma} \;d\mu \right]^{\frac{\gamma+q-1}{\gamma}} \\ & \times 
 \left[ \int_{\Omega} (\g \sigma)^{\frac{\gamma+q}{1-q}} \;d\sigma \right]^{\frac{(\gamma+q-1)(1-q)}{\gamma(\gamma+q)}},
\end{split}
\end{equation}
which is finite by \eqref{cond_sigma} and \eqref{cond_mu}, and hence \eqref{cond_mu-sigma} holds.

$\bullet$ Case 2: $\gamma + q <1$. Write
\[
\int_{\Omega} \left(\g\mu \right)^{\gamma+q}\;d\sigma
= \int_{\Omega} \left(\g\mu \right)^{\gamma+q} F^{a-1} F^{1-a}\;d\sigma,
\]
where $0<a<1$ and $F$ is a positive measurable function to be determined later. 
Applying H\"{o}lder's inequality with the exponents $\frac{1}{a}$ and $\frac{1}{1-a}$, we get
\begin{equation}\label{est1_case2}
\int_{\Omega} \left(\g\mu \right)^{\gamma+q}\;d\sigma
\leq \left( \int_{\Omega} \left(\g\mu \right)^{\frac{\gamma+q}{a}} F^{\frac{a-1}{a}}\;d\sigma 
\right)^{a} \left( \int_{\Omega} F \;d\sigma \right)^{1-a}.
\end{equation}
Setting $F=\left(\g\sigma\right)^{\frac{\gamma+q}{1-q}}$ and $a =\gamma+q$ in
\eqref{est1_case2}, we obtain
\begin{equation}\label{est2_case2}
\begin{split}
\int_{\Omega} \left(\g\mu \right)^{\gamma+q}\;d\sigma
& \leq \left( \int_{\Omega} \g\mu \left(\g\sigma\right)^{\frac{\gamma+q-1}{\gamma+q}}\;d\sigma 
\right)^{\gamma+q} \\ &\times \left( \int_{\Omega} \left(\g\sigma\right)^{\frac{\gamma+q}{1-q}} \;d\sigma 
\right)^{1-\gamma-q}.
\end{split}
\end{equation}
The first integral on the right-hand side of \eqref{est2_case2} is estimated by using Fubini's theorem, 
followed by inequality \eqref{iterated2} with 
$\omega= \sigma $ and $s =\frac{\gamma}{1-q}$, 
\begin{equation}\label{est3_case2}
\begin{split}
\int_{\Omega} \g\mu \left(\g\sigma\right)^{\frac{\gamma+q-1}{\gamma+q}}\;d\sigma
&= \int_{\Omega} \g\left((\g\sigma)^{\frac{\gamma+q-1}{\gamma+q}} d\sigma\right)\;d\mu\\
&\leq c \int_{\Omega} (\g\sigma)^{\frac{\gamma}{1-q}}\;d\mu.
\end{split}
\end{equation}
As above, we deduce
\begin{equation*}%\label{est4_case2}
\begin{split}
& \int_{\Omega} (\g\sigma)^{\frac{\gamma}{1-q}}\;d\mu
\leq \left( \int_{\Omega} \g\sigma (\g\mu)^{\gamma+q-1}\;d\mu\right)^{\frac{\gamma}{1-q}} 
\left( \int_{\Omega} (\g\mu)^{\gamma}\;d\mu \right)^{\frac{1-\gamma-q}{1-q}} \\
&= \left( \int_{\Omega} \g \left( (\g\mu)^{\gamma+q-1}d\mu\right)\;d\sigma \right)^{\frac{\gamma}{1-q}} 
\left( \int_{\Omega} (\g\mu)^{\gamma}\;d\mu \right)^{\frac{1-\gamma-q}{1-q}} \\
&\leq c \left( \int_{\Omega} (\g\mu)^{\gamma+q}\;d\sigma \right)^{\frac{\gamma}{1-q}} 
\left( \int_{\Omega} (\g\mu)^{\gamma}\;d\mu \right)^{\frac{1-\gamma-q}{1-q}}.
\end{split}
\end{equation*}
Combining the preceding estimates, we have
\begin{equation}\label{est5_case2}
\begin{split}
\left( \int_{\Omega} \left(\g\mu \right)^{\gamma+q}\;d\sigma \right)^{1-\frac{\gamma(\gamma+q)}{1-q}}
& \leq c \left[ \int_{\Omega} (\g \mu)^{\gamma} \;d\mu \right]^{\frac{(1-\gamma-q)(\gamma+q)}{1-q}}\\ & \times
 \left[ \int_{\Omega} (\g \sigma)^{\frac{\gamma+q}{1-q}} \;d\sigma \right]^{1-\gamma-q}.
 \end{split}
\end{equation}
This proves \eqref{cond_mu-sigma}, since both integrals on the right-hand side of \eqref{est5_case2} are finite by \eqref{cond_sigma} and \eqref{cond_mu}.

$\bullet$ Case 3: $\gamma +q =1$. Fix a positive number $\frac{1}{2-q}<a<1$. 
Applying H\"{o}lder's inequality with the exponents $\frac{1}{a}$ and $\frac{1}{1-a}$ we have
\begin{equation} \label{est1_case3}
\begin{split}
&\int_{\Omega} \g\mu \;d\sigma 
= \int_{\Omega} \g\mu (\g\sigma)^{\frac{a-1}{1-q}} (\g\sigma)^{\frac{1-a}{1-q}} \;d\sigma  \\
&\leq \left( \int_{\Omega} (\g\mu)^{\frac{1}{a}} (\g\sigma)^{\frac{a-1}{a(1-q)}}\;d\sigma \right)^{a}
         \left( \int_{\Omega} (\g\sigma)^{\frac{1}{1-q}} \;d\sigma \right)^{1-a}.
\end{split}
\end{equation}
We estimate the first integral on the right-hand side of \eqref{est1_case3} using inequalities \eqref{iterated1} and \eqref{iterated2}, together with Fubini's theorem and H\"{o}lder's inequality with the exponents
$\frac{a(1-q)}{1-a}$ and $\frac{a(1-q)}{a(2-q)-1}$, 
\begin{equation*}%\label{est2_case3}
\begin{split}
& \int_{\Omega} (\g\mu)^{\frac{1}{a}} (\g\sigma)^{\frac{a-1}{a(1-q)}}\;d\sigma 
\le c  \int_{\Omega} \g [(\g\mu)^{\frac{1-a}{a}}d \mu] (\g\sigma)^{\frac{a-1}{a(1-q)}}\;d\sigma \\
&= c \int_{\Omega} (\g\mu)^{\frac{1-a}{a}} \g \left( (\g\sigma)^{\frac{a-1}{a(1-q)}}\;d\sigma \right)\;d\mu \leq c \int_{\Omega} (\g\mu)^{\frac{1-a}{a}}  (\g\sigma)^{\frac{a-1}{a(1-q)}+1}\;d\mu \\
&\leq c \left( \int_{\Omega} (\g\mu)^{1-q} \;d\mu \right)^{\frac{1-a}{a(1-q)}}
            \left( \int_{\Omega} \g\sigma \;d\mu \right)^{\frac{a(2-q)-1}{a(1-q)}} \\
&= c \left( \int_{\Omega} (\g\mu)^{1-q} \;d\mu \right)^{\frac{1-a}{a(1-q)}}
            \left( \int_{\Omega} \g\mu \;d\sigma \right)^{\frac{a(2-q)-1}{a(1-q)}}.
\end{split}
\end{equation*}
Combining the preceding estimates, we deduce 
\begin{equation}\label{est3_case3}
\begin{split}
\left( \int_{\Omega} \g\mu \;d\sigma \right)^{1- \frac{a(2-q)-1}{(1-q)}}
& \leq c \left( \int_{\Omega} (\g\mu)^{1-q} \;d\mu \right)^{\frac{1-a}{(1-q)}}\\ & \times
\left( \int_{\Omega} (\g\sigma)^{\frac{1}{1-q}} \;d\sigma \right)^{1-a},
\end{split}
\end{equation}
which is finite by \eqref{cond_sigma} and \eqref{cond_mu}. Thus \eqref{cond_mu-sigma} holds.
\end{proof}

We are now prepared to show that conditions \eqref{cond_sigma} and \eqref{cond_mu} are
necessary and sufficient for the existence of a positive solution
$u \in L^{\gamma +q}(\Omega, d\sigma) \cap L^{\gamma}(\Omega, d\mu)$ 
to integral equation \eqref{eq_int}, under the same restrictions on the kernel $G$ as above.

\begin{Thm}\label{thm-main2}
Let $0<q<1$, $\gamma >0$, and let $\sigma, \mu \in \mathcal{M}^{+}(\Omega)$ 
with $\sigma \not\equiv 0$. Suppose $G$ is a positive, lower semicontinuous, 
quasi-symmetric  kernel 
on $\Omega \times \Omega$, which satisfies the WMP.
Suppose \eqref{cond_sigma} and \eqref{cond_mu} hold. Then there exists a positive solution 
$u \in L^{\gamma +q}(\Omega, d\sigma) \cap L^{\gamma}(\Omega, d\mu)$ to \eqref{eq_int}. The converse statement is also valid without the quasi-symmetry 
assumption on $G$.
\end{Thm}

\begin{proof} 
Suppose \eqref{cond_sigma} and \eqref{cond_mu} hold. Then, by Lemma \ref{lem_relation},
we see that \eqref{cond_mu-sigma} holds. Thus, in light of Theorem \ref{thm_int}, 
there exists a positive solution 
$u \in L^{\gamma +q}(\Omega, d\sigma)$ to \eqref{eq_int}.
We will show that $u \in L^{\gamma}(\Omega, d\mu)$ as well. By
\eqref{cond_mu}, it suffices to establish
\begin{equation}\label{cond_sigma-mu-2}
\int_{\Omega} \left[ \g(u^{q}d\sigma)\right]^{\gamma}\;d\mu < +\infty.
\end{equation}
Without loss of generality, we may assume that  $G$ is symmetric, and  $\mu \not\equiv 0$. Consider the following two cases:

$\bullet$ Case 1: $\gamma \geq 1$. Applying  \eqref{iterated2} with 
$\omega:= \sigma u^{q}$ and $s :=\gamma$, along with Fubini's theorem, and 
H\"{o}lder's inequality 
with the exponents $\frac{\gamma+q}{\gamma+q-1}$ and $\gamma +q$, we have 
\[
\begin{split}
& \int_{\Omega} \left[ \g(u^{q}d\sigma) \right]^{\gamma}\;d\mu
\leq c \int_{\Omega} \g \left( \left( \g (u^{q}d\sigma)\right)^{\gamma-1} u^{q}d\sigma \right) \;d\mu \\
&= c \int_{\Omega} \left( \g (u^{q}d\sigma)\right)^{\gamma-1}  (\g\mu) \, u^{q}d\sigma \\
&\leq c \left[ \int_{\Omega} \left(\g (u^{q}d\sigma)\right)^{\frac{(\gamma-1)(\gamma+q)}{\gamma+q-1}} u^{\frac{q(\gamma+q)}{\gamma+q-1}} \;d\sigma \right]^{\frac{\gamma+q-1}{\gamma+q}}
\left[ \int_{\Omega} (\g \mu)^{\gamma+q} \;d\sigma \right]^{\frac{1}{\gamma+q}}\\
&\leq c \left[ \int_{\Omega} u^{\gamma+q} \;d\sigma \right]^{\frac{\gamma+q-1}{\gamma+q}}
\left[ \int_{\Omega} (\g \mu)^{\gamma+q} \;d\sigma \right]^{\frac{1}{\gamma+q}} \\
&\leq c \int_{\Omega} u^{\gamma+q} \;d\sigma 
< +\infty.
\end{split}
\]

$\bullet$ Case 2: $0<\gamma<1$. We write
\[
\int_{\Omega} \left[ \g(u^{q}d\sigma) \right]^{\gamma}\;d\mu
= \int_{\Omega} \left[ \g(u^{q}d\sigma) \right]^{\gamma} F^{a-1} F^{1-a}\;d\mu,
\]
where $0<a<1$ and $F$ is a positive measurable function to be determined later. 
Applying H\"{o}lder's inequality with the conjugate exponents $\frac{1}{a}$ and $\frac{1}{1-a}$ yields
\begin{equation}\label{est3_thm-main2}
\begin{split}
\int_{\Omega} \left[ \g(u^{q}d\sigma) \right]^{\gamma}\;d\mu
&\leq \left[ \int_{\Omega} \g(u^{q}d\sigma) F^{\frac{a-1}{a}}  
 \;d\mu \right]^{a}
 \left[ \int_{\Omega} F
 \;d\mu \right]^{1-a}.
\end{split}
\end{equation}
Setting $F:=(\g\mu)^{\gamma}$ and $a:=\gamma$ in \eqref{est3_thm-main2}, we get
\begin{equation}\label{est4_thm-main2}
\begin{split}
 \int_{\Omega} \left[ \g(u^{q}d\sigma) \right]^{\gamma}\;d\mu
& \leq \left[ \int_{\Omega} \g(u^{q}d\sigma) (\g\mu)^{\gamma -1}  
 \;d\mu \right]^{\gamma}\\ & \times 
 \left[ \int_{\Omega} (\g\mu)^{\gamma}
 \;d\mu \right]^{1-\gamma}.
\end{split}
\end{equation}
We estimate the first integral on the right-hand side of \eqref{est4_thm-main2}  using Fubini's theorem, 
followed by inequality \eqref{iterated2} with $\omega= \mu$ and $s=\gamma$, 
\begin{equation}\label{est5_thm-main2}
\begin{split}
& \int_{\Omega} \g(u^{q}d\sigma) (\g\mu)^{\gamma -1}  
 \;d\mu
= \int_{\Omega} \g\left((\g\mu)^{r-q-1}d\mu\right) u^{q} \;d\sigma \\
&\leq c \int_{\Omega} (\g\mu)^{\gamma} u^{q} \;d\sigma \leq  c \int_{\Omega} u^{\gamma+q}  \;d\sigma < +\infty.
\end{split}
\end{equation}
By \eqref{est4_thm-main2} and \eqref{est5_thm-main2}, together with \eqref{cond_mu}, 
this proves \eqref{cond_sigma-mu-2}.

Conversely, since $u\in L^{\gamma}(\Omega, d\mu)$, it is clear that \eqref{cond_mu} holds.
Further, by Theorem \ref{thm_int}, condition \eqref{cond_sigma} is valid since
$u\in L^{\gamma +q}(\Omega, d\sigma)$.
\end{proof}

As an application of the preceding theorem, when $G$ is a positive Green function associated 
with $\LL$ in $\Omega$, we deduce the first part of Theorem \ref{thm-main} 
by appealing to the characterization of the generalized Green energy obtained in Sec. \ref{Sect:Energy}.

\begin{Thm}\label{thm_main3}
Let $0<q<1$, $\gamma>0$, and let $\sigma, \mu \in \mathcal{M}^{+}(\Omega)$
with $\sigma \not\equiv 0$. Let $G$ be a positive Green function associated with $\LL$ on $\Omega$.
Then there exists a positive solution $u \in L^{q}_{loc}(\Omega, d\sigma)$ 
to \eqref{main_problem} satisfying \eqref{gamma-norm}
if and only if \eqref{cond_sigma} and \eqref{cond_mu} hold.

In this case, $u$ is a minimal solution in the sense that $u\leq v$ q.e. 
for any positive solution $v \in L^{q}_{loc}(\Omega)$ to \eqref{main_problem} which
satisfies \eqref{gamma-norm}.
\end{Thm}

\begin{proof}
This follows from Theorem \ref{Thm:Energy} with $\omega:= u^{q}d\sigma + \mu$, 
together with Theorem \ref{thm-main2}. 
Moreover, arguing by induction, we see that minimality of such a solution 
follows immediately from its construction in Theorem \ref{thm_int} (cf. \cite[Lemma 5.5]{SV}).
\end{proof}

\begin{Rem}\label{Rem_Thm}
When $\gamma =1$, uniqueness of such a solution in $u \in \dot{W}^{1,2}_{0}(\Omega)$, 
follows from a similar argument presented in \cite[Sec. 6]{SV} (see also \cite{CV}), 
using its minimality together with a convexity property 
of the Dirichlet integral $\int_{\Omega}|\nabla u|^{2}\;dx$, which is comparable to 
the expression $\int_{\Omega} \A\nabla u \cdot \nabla u \;dx$ due to the uniform ellipticity condition 
\eqref{cond:ellipticity}.
\end{Rem}

Applying the next theorem with $\omega:=u^{q}d\sigma + \mu$ yields Corollary \ref{cor-main}.

\begin{Thm}\label{below-energy}
Let $n\geq 3$, and $\omega \in \M^{+}(\Omega)$ with $\omega \not\equiv 0$. Let $G$ be a positive Green function associated with $\LL$ on $\Omega$. Suppose that $u:=\g\omega$ satisfies \eqref{gamma-norm} for some $0<
\gamma\le 1$. Then $u \in \dot{W}_{0}^{1,p}(\Omega)$ where  
$p=\frac{n(1+\gamma)}{n+\gamma-1}$. 
If, in addition, $|\Omega|<+\infty$, then the assertion is valid for 
$ 1 \leq  p \leq \frac{n(1+\gamma)}{n+\gamma-1}$.
\end{Thm}
\begin{proof} The classical case $\gamma=1$ is known, so we can assume $\gamma \in (0, 1)$. 
Observe that $u$ is a positive $\A$-superharmonic function with 
zero boundary values, and $p=\frac{n(1+\gamma)}{n+\gamma-1} \in (\frac{n}{n-1}, 2)$. For each $k \in \N$, set 
$u_{k}=\min(u,k)$, which is a positive $\A$-superharmonic function of the class 
$L^{\infty}(\Omega)\cap W_{loc}^{1,2}(\Omega)$. 
Let $\lbrace \Omega_{k} \rbrace_{k=1}^{\infty}$ be an increasing sequence of 
relatively compact open subsets of $\Omega$ such that 
$\Omega = \bigcup_{k=1}^{\infty} \Omega_{k}$.
Applying H\"{o}lder's inequality with the exponents
$ \frac{2}{p}$ and  $\frac{2}{2-p}$, followed by Sobolev's inequality \cite[Theorem 1.56]{MZ}, we obtain
\begin{equation*}%\label{est1-below-energy}
\begin{split}
& \|(\nabla u_{k}) \chi_{\Omega_k} \|_{L^{p}(\Omega)}^{p}
= \int_{\Omega_{k}} |\nabla u_{k}|^{p} \;dx  \\
&= \int_{\Omega_{k}}  |\nabla u_{k}|^{p} u_{k}^{\frac{(\gamma-1)p}{2}} 
u_{k}^{\frac{(1-\gamma)p}{2}}\;dx \\
&\leq \left( \int_{\Omega_{k}}| \nabla u_{k}|^{2}u_{k}^{\gamma-1}\;dx \right)^{\frac{p}{2}}
\left( \int_{\Omega_{k}} u_{k}^{\frac{(1-\gamma)p}{2-p}}\;dx \right)^{\frac{2-p}{2}} \\ 
&\leq \left( \int_{\Omega}| \nabla u|^{2}u^{\gamma-1}\;dx \right)^{\frac{p}{2}}
\|u_{k}\|_{L^{\frac{(1-\gamma)p}{2-p}}(\Omega_{k})}^{\frac{(1-\gamma)p}{2}}\\
&\leq c \left( \int_{\Omega}| \nabla u|^{2}u^{\gamma-1}\;dx \right)^{\frac{p}{2}}
\|\nabla u_{k} \|_{L^{p}(\Omega_{k})}^{\frac{(1-\gamma)p}{2}},
\end{split}
\end{equation*}
that is,
\begin{equation}\label{est2-below-energy}
\|(\nabla u_{k}) \chi_{\Omega_k} \|_{L^{p}(\Omega)}^{p -\frac{(1-\gamma)p}{2}} 
\leq c \left( \int_{\Omega}| \nabla u|^{2}u^{\gamma-1}\;dx \right)^{\frac{p}{2}},
\end{equation}
where $c$ is a positive constant independent of $k$. Since $p > \frac{(1-\gamma)p}{2}$, letting
$k \rightarrow \infty$ in \eqref{est2-below-energy} yields the assertion by  the monotone convergence theorem. This proves $u \in \dot{W}_{0}^{1,p}(\Omega)$, 
which  is obviously  true for all 
$1 \leq  p \leq \frac{n(1+\gamma)}{n+\gamma-1}$ when $|\Omega|<+\infty$.\end{proof}

The next proposition shows in particular that a pair of conditions in
\eqref{cond:pair1} is sufficient for both 
\eqref{cond_sigma} and \eqref{cond_mu}.

\begin{Prop}\label{Prop}
Let $G$ be a positive lower semicontinuous kernel on $\Omega\times \Omega$ which satisfies
\begin{equation}\label{cond_G}
G(x,y) \leq c \, I_{2\alpha}(x-y), \quad \forall x,y \in \Omega,
\end{equation} 
where $I_{2\alpha}(\cdot) = |\cdot|^{2\alpha-n} $ is the Riesz kernel of order $2\alpha$  $(0<\alpha<\frac{n}{2})$ on $\R^n$, and $c$ is a positive constant. 
Let $\beta>0$. If $\omega \in L^{s}(\Omega)$ is a positive function, where 
$
s=\frac{n(\beta+1)}{n+2\alpha\beta},
$
then
\begin{equation}\label{cond_general}
\g\omega \in L^{\beta}(\Omega, d\omega).
\end{equation}
\end{Prop}

\begin{proof}
Observe that $s=\frac{n(\beta+1)}{n+2\alpha\beta}>\frac{n(\beta+1)}{n+n\beta}=1$. 
Then
\begin{equation}\label{est1_Prop1}
\int_{\Omega} (\g\omega)^{\beta}\;d\omega
\leq \left(\int_{\Omega} (\g\omega)^{\beta s'}\,dx\right)^{\frac{1}{s'}} \| \omega \|_{L^{s}(\Omega)},
\end{equation}
where $s'=\frac{s}{s-1}$ is the H\"{o}lder conjugate of $s$. Denote $\tilde{\omega}$ the zero 
extension of $\omega$ to $\R^n$. By  
\eqref{cond_G} and the Hardy-Littlewood-Sobolev inequality, there is a positive constant $C\geq c$ such that
\begin{equation}\label{est2_Prop1}
\| \g\omega \|_{L^{\beta s'}(\Omega)}
\leq C \| \I_{2\alpha}\tilde{\omega}\|_{L^{\beta s'}(\R^n)} 
\leq C \| \tilde{\omega}\|_{L^{s}(\R^n)} 
= C \| \omega \|_{L^{s}(\Omega)}.
\end{equation}
Combining \eqref{est1_Prop1} and \eqref{est2_Prop1} yields \eqref{cond_general}.
\end{proof}

Proposition \ref{Prop} shows that \eqref{cond:pair2} implies 
\eqref{cond:pair0}. Hence,  Corollary \ref{cor-main2} follows from  Corollary \ref{cor-main}.
%%%%%%%%%%%%%%%%%%%%%%%%%%%%%%%%%%%%
%%%%%%%%%%%%%%%%%%%%%%%%%%%%%%%%%%%%
%%%%%%%%%%%%%%%%%%%%%%%%%%%%%%%%%%%%
%%%%%%%%%%%%%%%%%%%%%%%%%%%%%%%%%%%%

%%% ----------------------------------------------------------------------------------------------
%%% ----------------------------------------------------------------------------------------------
\end{document}